\documentclass[a4paper,leqno]{amsart}
\usepackage[utf8]{inputenc}
\usepackage[T1]{fontenc}

\usepackage{enumerate,enumitem}
\usepackage{lscape}
\usepackage[usenames,dvipsnames]{xcolor}
\usepackage{mathabx,amsmath,amssymb,stmaryrd,mathtools} 
\usepackage{pstricks,pst-node,pst-coil,pst-plot}

\numberwithin{equation}{section}

\usepackage[colorlinks=true, bookmarks=true, pdfstartview=FitH, pagebackref=true]{hyperref}
\usepackage[space, compress, sort]{cite}
\usepackage{verbatim}

\theoremstyle{plain}

\newtheorem{thm}{Theorem}[section]
\newtheorem{theorem}[thm]{Theorem}

\newtheorem{lemma}[thm]{Lemma}

\newtheorem{proposition}[thm]{Proposition}

\theoremstyle{remark}

\theoremstyle{definition}

\newcommand{\bE}{\mathbb{E}}
\newcommand{\bR}{\mathbb{R}}
\newcommand{\bS}{\mathbb{S}}

\newcommand{\omegatil}{\widetilde{\omega}}
\newcommand{\ftil}{\widetilde{f}}

\DeclareMathOperator{\diam}{diam}
\DeclareMathOperator{\conv}{conv}
\DeclareMathOperator{\vol}{Vol}

\date{\today \ at \currenttime}

\begin{document}


\author[R. Gicquaud]{Romain Gicquaud}
\address[R. Gicquaud]{Institut Denis Poisson \\ UFR Sciences et Technologie \\
    Facult\'e de Tours \\ Parc de Grandmont\\ 37200 Tours \\ FRANCE}
\email{\href{mailto: R. Gicquaud <Romain.Gicquaud@idpoisson.fr>}{romain.gicquaud@idpoisson.fr}}

\title[Poincaré-Sobolev inequalities]{A note on Poincaré-Sobolev type inequalities on compact manifolds}

\begin{abstract}
    We prove a Poincaré–Sobolev type inequality on compact Riemannian manifolds where the deviation of a function from a biased average, defined using a density $\omega$, is controlled by the unweighted $L^p$-norm of its gradient. Unlike classical weighted Poincaré inequalities, the density does not enter the measure or the Sobolev norms, but only the reference average. We show that the associated Poincaré constant depends quantitatively on $\|\omega\|_{L^q}$. This framework naturally arises in the analysis of coupled elliptic systems and seems not to have been addressed in the existing literature.
\end{abstract}

\date{\today}
\maketitle

\tableofcontents

\section{Introduction}
Inequalities of Poincaré and Sobolev type are among the most fundamental tools
in the analysis of partial differential equations. They provide quantitative
control of the oscillation of a function in terms of its derivatives and lie at
the heart of \emph{a priori} estimates, compactness arguments, and regularity
theory for elliptic and parabolic equations.

In its simplest form, the Poincaré inequality asserts that, on a compact
Riemannian manifold $(M,g)$, the deviation of a function from its mean value is
controlled by the $L^p$-norm of its gradient: for any $p \in [1,\infty)$, there
exists a constant $C = C(M,g,p)$ such that, for any function $f \in
    W^{1,p}(M,\bR)$,
\[
    \left\| f - \bE[f] \right\|_{L^p}
    \leq C \|df\|_{L^p},
\]
where $\bE[f]$ denotes the average value of $f$,
\[
    \bE[f]
    = \frac{1}{\operatorname{vol}(M,g)} \int_M f\,d\mu^g.
\]
We refer the reader to~\cite[Section~5.8]{Evans} for a proof of this inequality
on bounded domains of $\bR^n$ and to~\cite[Lemma~3.8]{Hebey} for compact
manifolds. More refined versions, based on the Sobolev inequality, allow one to
estimate this deviation in stronger Lebesgue norms and are indispensable in the
study of nonlinear elliptic equations.

In many applications, however, the reference average is not the uniform mean
with respect to the Riemannian volume measure. Instead, one is led to consider
averages defined using a density $\omega$,
\[
    \bE_\omega[f] = \int_M f\,\omega\,d\mu^g,
\]
where $\omega$ is a non-negative function normalized by $\int_M \omega\,d\mu^g
    = 1$. Such weighted averages arise naturally in problems involving
normalization constraints, conservation laws, or coupled systems of equations,
where the density $\omega$ may depend on other unknowns of the system.

For a fixed density $\omega$ and $p \in (1,\infty)$, it is not difficult to
prove a Poincaré-type inequality controlling $f - \bE_\omega[f]$ by
$\|df\|_{L^p}$. A standard compactness argument analogous
to~\cite[Lemma~3.8]{Hebey} shows that there exists a constant $C =
    C(M,g,p,\omega)$ such that
\[
    \|f - \bE_\omega[f]\|_{L^p}
    \leq C \|df\|_{L^p}
    \qquad \text{for all } f \in W^{1,p}(M,\bR).
\]
However, this argument provides no information on how the constant depends on
$\omega$, and in particular does not allow one to obtain uniform estimates when
the density varies in a family of weights.

Understanding this dependence is a subtle issue. To our knowledge, the only
case in which such an estimate is explicitly addressed in the literature is
when $\omega$ is the characteristic function of a measurable subset of a
bounded convex domain, see~\cite[Lemma 7.16]{GilbargTrudinger}, a situation
related to nonlinear capacity theory (see, for instance,~\cite{MazyaSobolev}).
Another source of interest in Poincaré-type inequalities involving densities
comes from the theory of Bakry--Émery curvature-dimension conditions (see,
e.g.,~\cite{LiWang}); in that framework, however, the density defines the
underlying measure itself and modifies the geometric structure of the space, so
that the resulting inequalities belong to a fundamentally different setting.

The aim of this paper is to establish a Poincaré--Sobolev type inequality in
which the deviation of a function from a weighted average is controlled by the
unweighted Sobolev norm of its gradient, with an explicit quantitative
dependence of the constant on the density. More precisely, we prove the
following result.

\begin{theorem}\label{thmPoincare}
    Let $(M,g)$ be a compact Riemannian manifold of dimension $n$, and let
    $p,q \in (1,\infty)$ satisfy
    \[
        p \geq \frac{n}{n-1}
        \quad \text{and} \quad
        q > \frac{n}{2}.
    \]
    Let $r \in (1,\infty]$, and assume that, if $p < n$, then
    \[
        \frac{1}{r} \geq \frac{1}{p} - \frac{1}{n}.
    \]
    Then, for all functions $f \in W^{1,p}(M,\bR)$ and all non-negative densities
    $\omega \in L^q(M)$ normalized by $\bE[\omega] = 1$, the weighted average
    $\bE_\omega[f]$ is well-defined and there exists a constant $C>0$, depending
    only on $(M,g,p,q,r)$, such that
    \[
        \left\| f - \bE_\omega[f] \right\|_{L^r}
        \leq
        C\,\|\omega\|_{L^q}^{\frac{n}{(n-1)p}}\,\|df\|_{L^p}.
    \]
\end{theorem}

To the best of our knowledge, such a quantitative estimate has not previously
been established, even in the Euclidean setting.

The outline of the paper is as follows. In Section~\ref{secRn}, we address the
case of bounded convex open subsets of $\bR^n$. In Section~\ref{secDiffeo}, we
construct a local diffeomorphism $\Psi: B \to M$ from the open unit ball $B$ in
$\bR^n$ onto the manifold $M$. Finally, in Section~\ref{secProof}, we use the
coarea formula to lift the function and $\omega$ from $M$ to $B$ to prove our
main result.

\subsection*{Acknowledgments}

I am grateful to Laurent Véron for useful discussion about the paper.

\section{The case of open subsets of \texorpdfstring{$\bR^n$}{TEXT}}\label{secRn}
We start with the fundamental case of the inequality. The construction is
inspired by~\cite[Lemma 7.16]{GilbargTrudinger}:

\begin{lemma}\label{lmPoincare1}
    Let $U$ be an open non-empty bounded convex subset of $\bR^n$ with smooth boundary. Let $p, q \in (1, \infty)$ be such that $p \geq \frac{n}{n-1}$ and $q > \frac{n}{2}$. There exists a constant $c = c(U, p, q)$ such that, for any $\omega \in L^q(U, \bR)$, $\omega \geq 0$ a.e., $\bE[\omega] = 1$, and any function $f \in W^{1, p}(U, \bR)$,
    \[
        \int_U \left|f(x) - \bE_\omega[f]\right|^t dx \leq c \|\omega\|_{L^q} \left(\int_U |df|^p dx\right)^{\frac{t}{p}},
    \]
    where $t = \frac{n-1}{n}p$.
\end{lemma}

\begin{proof}
    We prove the inequality for $C^1$ functions $u$. The general case follows by a density argument. Let $x \in U$ be given. Then, for any $y \in U$, $y \neq x$, setting $\theta \coloneq \frac{y - x}{|y-x|}$, we have
    \[
        f(y) - f(x)
        = \int_0^{|y-x|} \frac{d}{ds}\left(f(x + s \theta)\right) ds
        = \int_0^{|y-x|} df_{x + s \theta} (\theta) ds.
    \]
    Upon multiplying by $\omega(y)$ and integrating over $U$, we have
    \begin{align*}
        \bE_\omega[f] - f(x)
         & = \int_U \omega(y) (f(y) - f(x)) dy                                 \\
         & = \int_U \omega(y) \int_0^{|y-x|} df_{x + s \theta} (\theta) ds dy.
    \end{align*}
    As $\omega(y) dy$ is a probability measure and $t \geq 1$, we can use Jensen's inequality to get
    \begin{align*}
        |f(x) - \bE_\omega[f]|^t
         & \leq \int_U \omega(y) \left|\int_0^{|y-x|} df_{x + s \theta} (\theta) ds\right|^t dy              \\
         & \leq \int_U \omega(y) |y-x|^{t-1} \int_0^{|y-x|} \left|df_{x + s \theta} (\theta)\right|^t ds dy,
    \end{align*}
    where we used a second time Jensen's inequality to get
    \[
        \left|\frac{1}{|y-x|} \int_0^{|y-x|} df_{x + s \theta} (\theta) ds\right|^t
        \leq \frac{1}{|y-x|} \int_0^{|y-x|} \left|df_{x + s \theta} (\theta)\right|^t ds,
    \]
    as $\frac{ds}{|y-x|}$ is, once again, a probability measure on the interval
    $[0, |y-x|]$.

    We now pass to polar coordinates for $y-x$. Note that $\omega$ can be extended
    to zero outside of $U$. We also extend $df$ by zero outside of $U$ without
    changing the notation. This allows us not to care about restricting the domain
    of integration. Let also $d = \diam(U)$. We have
    \[
        |f(x) - \bE_\omega[f]|^t
        \leq \int_{r=0}^d \int_{\bS^{n-1}} r^{n+t-2} \omega(x + r \theta) \int_0^r |df_{x + s \theta}|^t ds d\theta dr,
    \]
    where $d\theta$ is the surface measure on the unit sphere $\bS^{n-1} \subset
        \bR^n$. Hence,
    \[
        |f(x) - \bE_\omega[f]|^t
        \leq \int_{0 \leq s \leq r \leq d} \int_{\bS^{n-1}} r^{n+t-2} \omega(x + r \theta) |df_{x + s \theta}|^t d\theta dr ds.
    \]
    Integrating against $x \in U$, we get
    \begin{align*}
        \int_U |f(x) - \bE_\omega[f]|^t dx
         & \leq \int_U \int_{0 \leq s \leq r \leq d} \int_{\bS^{n-1}} r^{n+t-2} \omega(x + r \theta) |df_{x + s \theta}|^t d\theta dr ds dx                      \\
         & \leq \int_U \int_{\bS^{n-1}} \int_{s=0}^d \left(\int_{s \leq r \leq d} r^{n+t-2} \omega(x + r \theta) dr \right) |df_{x + s \theta}|^t d\theta ds dx  \\
         & \leq d^{n+t-2} \int_U \int_{\bS^{n-1}} \int_{s=0}^d \left(\int_{s \leq r \leq d} \omega(x + r \theta) dr \right) |df_{x + s \theta}|^t d\theta ds dx.
    \end{align*}
    We change the order of integration and perform the change of variable $z = x + s \theta$ to get:
    \begin{align}
        \int_U |f(x) - \bE_\omega[f]|^t dx
         & \leq d^{n+t-2} \int_{\bS^{n-1}} \int_{s=0}^d \int_{x \in U} \left(\int_{s \leq r \leq d} \omega(x + r \theta) dr \right) |df_{x + s \theta}|^t dx d\theta ds\nonumber \\
         & = d^{n+t-2} \int_{\bS^{n-1}} \int_{s=0}^d \int_{z \in U} \left(\int_{s \leq r \leq d} \omega(z + (r-s) \theta) dr \right) |df_{z}|^t dz d\theta ds\nonumber           \\
         & = d^{n+t-2} \int_{z \in U} \left(\int_{\bS^{n-1}} \int_{s=0}^d \int_{s \leq r \leq d} \omega(z + (r-s) \theta) dr ds d\theta\right) |df_{z}|^t dz\nonumber            \\
         & \leq d^{n+t-1} \int_{z \in U} \left(\int_{\bS^{n-1}} \int_{\rho=0}^d \omega(z + \rho \theta) d\rho d\theta\right) |df_{z}|^t dz,\label{eqPoincare3}
    \end{align}
    where we also performed the change of variable $\rho = r-s$ to obtain the last line. Now remark that the inner integral can be rewritten as follows:
    \begin{align*}
        \omegatil(z)
         & = \int_{\bS^{n-1}} \int_{t=0}^d \omega(z + \rho \theta) d\rho d\theta                        \\
         & = \int_{\bS^{n-1}} \int_{t=0}^d \omega(z + \rho \theta) \rho^{-n-1} \rho^{n-1} d\rho d\theta \\
         & = \int_{\bR^n} \omega(z - x) \frac{\chi_{|x| \leq d}}{|x|^{n-1}} dx,
    \end{align*}
    where we have set $x = \rho \theta$. Note that the function $x \mapsto \frac{\chi_{|x| \leq d}}{|x|^{n-1}}$ belongs to $L^r(\bR^n, \bR)$ for any $r < \frac{n}{n-1}$\footnote{This function is actually a truncated Riesz potential for which we could use the Hardy-Littlewood-Sobolev inequality~\cite[Section 7.8]{GilbargTrudinger}. This is not needed here.} so we use Young's inequality together with the fact that $q > \frac{n}{2}$ to get
    \[
        \|\omegatil\|_{L^n} \leq c \|\omega\|_{L^q}
    \]
    for some constant $c = c(n, q, d)$. We finally conclude
    from~\eqref{eqPoincare3} that
    \[
        \int_U |f(x) - \bE_\omega[f]|^t dx \leq c d^{n+p-2} d^{n+t-1} \int_{z \in U} \omegatil(z) |df_{z}|^t dz \leq \|\omegatil\|_{L^n} \left(\int_U |df|^p dx\right)^{\frac{n-1}{n}}.
    \]
\end{proof}

The next step is to extend the previous result to a broader range for $t$ by
means of the Sobolev inequality:
\begin{lemma}\label{lmPoincare2}
    Let $U$ be an open non-empty bounded convex subset of $\bR^n$ with smooth boundary. Let $p, q \in (1, \infty)$ be such that $p \geq \frac{n}{n-1}$ and $q > \frac{n}{2}$. Let $r \in [1, \infty]$ be such that
    \[
        \frac{1}{r} \geq \frac{1}{p} - \frac{1}{n} \quad\text{if } p < n.
    \]
    There exists a constant $c' = c'(U, p, q, r)$ such that, for any $\omega \in
        L^q(U, \bR)$, $\omega \geq 0$ a.e., $\bE[\omega] = 1$ and any function $f \in
        W^{1, p}(U, \bR)$,
    \[
        \|f - \bE_\omega[f]\|_{L^r} \leq c' \|\omega\|_{L^q}^{\frac{n}{(n-1)p}} \|df\|_{L^p}.
    \]
\end{lemma}

\begin{proof}
    Let $t = \frac{n-1}{n} p$ be as in Lemma~\ref{lmPoincare1}. The conclusion of this lemma can be written as follows:
    \begin{equation}\label{eqPoincare4}
        \|f - \bE_\omega[f]\|_{L^t} \leq (c \|\omega\|_{L^q})^{\frac{n}{(n-1)p}} \|df\|_{L^p}.
    \end{equation}
    As $U$ has finite volume, the conclusion of the lemma is immediate if $r \leq t$ as there is a continuous inclusion $L^t(U, \bR) \hookrightarrow L^r(U, \bR)$. If, instead, $r > t$, we let $p'$ be defined as follows:
    \[
        p' =
        \left\lbrace
        \begin{aligned}
            \frac{np}{n-p}       & \quad\text{if } p < n, \\
            \infty               & \quad\text{if } p > n, \\
            \text{arbitrary} > p & \quad\text{if } p = n.
        \end{aligned}
        \right.
    \]
    Let $\ftil = f - \bE_\omega[f]$. The Sobolev embedding theorem implies that
    there exists a constant $s > 0$ such that
    \begin{equation}\label{eqPoincare5}
        \|\ftil\|_{L^{p'}} \leq s \left(\|d\ftil\|_{L^p} + \|\ftil\|_{L^p}\right).
    \end{equation}
    As $t < p < p'$, we have
    \[
        \|\ftil\|_{L^p} \leq \|\ftil\|_{L^t}^\theta \|\ftil\|_{L^{p'}}^{1-\theta}\quad\text{with } \theta \in (0, 1) \text{ such that } \frac{1}{p} = \frac{\theta}{t} + \frac{1-\theta}{p'}.
    \]
    By the $\epsilon$-Young inequality~\cite[Equation $(7.6)$]{GilbargTrudinger},
    we deduce that, for any $\epsilon > 0$,
    \[
        \|\ftil\|_{L^p} \leq \theta \epsilon^{-\frac{1-\theta}{\theta}} \|\ftil\|_{L^t} + (1-\theta) \epsilon \|\ftil\|_{L^{p'}}.
    \]
    From~\eqref{eqPoincare4}, we deduce that
    \[
        \|\ftil\|_{L^p} \leq \theta \epsilon^{-\frac{1-\theta}{\theta}} (c \|\omega\|_{L^q})^{\frac{n}{(n-1)p}} \|df\|_{L^p} + (1-\theta) \epsilon \|\ftil\|_{L^{p'}}.
    \]
    Hence,~\eqref{eqPoincare5} implies
    \[
        \|\ftil\|_{L^{p'}} \leq s \left(\|d\ftil\|_{L^p} + \theta \epsilon^{-\frac{1-\theta}{\theta}} (c \|\omega\|_{L^q})^{\frac{n}{(n-1)p}} \|df\|_{L^p} + (1-\theta) \epsilon \|\ftil\|_{L^{p'}}\right),
    \]
    which we rewite as follows:
    \[
        \left(\frac{1}{s} - (1-\theta) \epsilon\right) \|\ftil\|_{L^{p'}} \leq \left(1 + \theta \epsilon^{-\frac{1-\theta}{\theta}} (c \|\omega\|_{L^q})^{\frac{n}{(n-1)p}}\right) \|d\ftil\|_{L^p}.
    \]
    Choosing $\epsilon$ such that $(1-\theta) \epsilon = \frac{1}{2s}$, we conclude
    that
    \[
        \|\ftil\|_{L^{p'}} \leq 2s \left(1 + \theta \epsilon^{-\frac{1-\theta}{\theta}} (c \|\omega\|_{L^q})^{\frac{n}{(n-1)p}}\right) \|d\ftil\|_{L^p}.
    \]
    As $1 = \bE[\omega] = \|\omega\|_{L^1} \leq \vol(U)^{1 - \frac{1}{q}}
        \|\omega\|_{L^q}$, we can replace the constant $1$ in the previous inequality
    by
    \[
        1 = 1^{\frac{n}{(n-1)p}} \leq \vol(U)^{\left(1 - \frac{1}{q}\right)\frac{n}{(n-1)p}} \|\omega\|_{L^q}^{\frac{n}{(n-1)p}}.
    \]
    Thus, we conclude that there is a constant $c'$ depending only on $U$ and $p$
    but not on $f$ such that
    \[
        \|\ftil\|_{L^{p'}} \leq c' \|\omega\|_{L^q}^{\frac{n}{(n-1)p}} \|d\ftil\|_{L^p},
    \]
    i.e.\ such that
    \[
        \|f - \bE_\omega[f]\|_{L^{p'}} \leq c' \|\omega\|_{L^q}^{\frac{n}{(n-1)p}} \|df\|_{L^p}.
    \]
    Using once again that, if $r \leq p'$, we have a continuous embedding
    $L^{p'}(U, \bR) \hookrightarrow L^r(U, \bR)$, we obtain the conclusion of the
    lemma.
\end{proof}

\section{Local diffeomorphisms onto compact manifolds}\label{secDiffeo}

The strategy for extending the Poincaré inequality from convex subsets of
$\bR^n$ to compact manifolds relies on constructing a surjective local
diffeomorphism from an open ball $B \subset \bR^n$ onto a given compact
Riemannian manifold $(M,g)$. This is the content of the following proposition:

\begin{proposition}\label{propDiffeo}
    Let $(M^n, g)$ be a connected compact Riemannian manifold. There exists a surjective local diffeomorphism $\Psi: B \to M$, where $B$ is the unit ball in $\bR^n$, and a constant $K > 0$ such that, for any $x \in M$, the fiber $\Psi^{-1}(x)$ has cardinal at most $K$.
\end{proposition}

Although the existence of such a diffeomorphism does not depend on the choice
of a Riemannian metric on $M$, we include the metric in the statement since it
is relevant to our context, and the construction we use explicitly requires
one.

\begin{proof}
    The first step is inspired by~\cite[Theorem 5.1]{BottTu}. Let $\conv_g(M) > 0$ denote the convexity radius of $(M, g)$, i.e., the largest $r > 0$ such that, for any $x \in M$, the distance function $d_g(x, \cdot)$ is convex on the geodesic ball $B_g(x, r)$ (see e.g.~\cite[Chapter 6]{Petersen}).

    Since $M$ is compact, it can be covered by finitely many such convex balls. Let
    $\{x_i\}_{i = 1}^K \subset M$ be a finite set of points such that the balls
    $B_i \coloneq B_g(x_i, \conv_g(M))$ cover $M$. Define the incidence graph $G$
    of this covering: its vertices are $\{1, \ldots, K\}$, with an edge between $i$
    and $j$ if and only if $B_i \cap B_j \neq \emptyset$. Since $M$ is connected,
    $G$ is connected, and therefore admits a spanning tree $T \subset G$
    (see~\cite[Section 1.5]{Diestel}).

    Let $B \subset \bR^n$ denote the open unit ball, and let $\phi_i: B \to B_i
        \subset M$ be a diffeomorphism, for example given by a rescaled exponential map
    centered at $x_i$. Then $M$ can be reconstructed by gluing together the balls
    $B$ via the identifications:
    \[
        M \cong \left( \bigsqcup_{i = 1}^K \{i\} \times B \right) \big/ \!\sim,
    \]
    where the equivalence relation $\sim$ is defined by:
    \[
        (i, x) \sim (j, y) \quad \Leftrightarrow \quad
        \begin{cases}
            i = j                   & \text{ and } x = y,                 \\
                                    & \text{or}                           \\
            i \neq j,\ (i, j) \in G & \text{ and } \phi_i(x) = \phi_j(y).
        \end{cases}
    \]

    Note that the condition $(i, j) \in G$ is redundant, since $\phi_i(x) =
        \phi_j(y)$ implies $B_i \cap B_j \neq \emptyset$. However, we may restrict the
    identifications to edges $(i, j) \in T$, the chosen spanning tree. This yields
    a smooth open manifold $\widetilde{M}$ defined by:
    \[
        \widetilde{M} \coloneq \left( \bigsqcup_{i = 1}^K \{i\} \times B \right) \big/ \!\sim_T,
    \]
    where $(i, x) \sim_T (j, y)$ if and only if:
    \[
        i = j \text{ and } x = y, \quad\text{or} \quad (i, j) \in T \text{ and } \phi_i(x) = \phi_j(y).
    \]

    Since $T$ is a tree, the space $\widetilde{M}$ is a smooth, connected, open
    manifold of dimension $n$, and the natural projection $\pi : \widetilde{M} \to
        M$ is a surjective local diffeomorphism.

    We now show that $\widetilde{M}$ is diffeomorphic to the unit ball $B \subset
        \bR^n$, by induction on the number $K$ of vertices in $T$.

    The base case $K = 1$ is trivial, as $\widetilde{M} = B$. Assume the statement
    holds for a tree with $K-1$ vertices, and let $T$ have $K$ vertices. Choose a
    leaf $i$ of $T$, and let $j$ be its unique neighbor. Let $\widetilde{N}$ be the
    open manifold obtained by removing the chart $\{i\} \times B$ from the disjoint
    union and performing the same identifications as in $\widetilde{M}$, restricted
    to $T' = T \setminus \{(i,j)\}$. By induction, $\widetilde{N} \cong B$.

    We now show that $\widetilde{M} \cong \widetilde{N}$. Observe that $B_i \cup
        B_j \subset M$ is open, contains a convex subset $B_i \cap B_j$, and is
    therefore star-shaped with respect to any point $x \in B_i \cap B_j$. Hence,
    $B_i \cup B_j$ is diffeomorphic to $B_j$. We may choose a diffeomorphism $\psi
        : B_i \cup B_j \to B_j$ such that $\psi$ is the identity on the overlaps $B_j
        \cap B_k$ for any neighbor $k$ of $j$ in $T'$. This diffeomorphism can be
    extended to all of $\widetilde{M}$ by acting as the identity elsewhere, and
    thus defines a diffeomorphism $\widetilde{M} \cong \widetilde{N}$.

    By induction, we obtain a global diffeomorphism $\Phi : \widetilde{M} \to B$.
    Composing with the local diffeomorphism $\pi: \widetilde{M} \to M$, we obtain
    the required surjective local diffeomorphism:
    \[
        \Psi \coloneq \pi \circ \Phi^{-1}: B \twoheadrightarrow M.
    \]
\end{proof}

\section{Proof of the main result}\label{secProof}

We can now prove the main result of this note:

\begin{proof}[Proof of Theorem~\ref{thmPoincare}]
    From Proposition~\ref{propDiffeo}, there exists a surjective local diffeomorphism $\Psi : B \to M$, where $B \subset \bR^n$ is the open unit ball. For each $\epsilon \in (0, 1)$, let $B_{1 - \epsilon} \subset B$ denote the open ball of radius $1 - \epsilon$ centered at the origin. Then
    \[
        M = \bigcup_{\epsilon \in (0, 1)} \Psi(B_{1 - \epsilon}),
    \]
    and since $M$ is compact and the sets $\Psi(B_{1-\epsilon})$ are decreasing
    with respect to $\epsilon$, there exists $\epsilon_0 \in (0,1)$ such that $M =
        \Psi(B_{1 - \epsilon_0})$.

    As $B_{1 - \epsilon_0}$ is relatively compact in $B$, the pullback metric
    $\widetilde{g} = \Psi^* g$ is smooth and uniformly equivalent to the Euclidean
    metric $\delta$ on $B_{1 - \epsilon_0}$: there exists a constant $C_0 \geq 1$
    such that
    \begin{equation}\label{eqEquivalentMetrics}
        C_0^{-1} \delta \leq \widetilde{g} \leq C_0 \delta.
    \end{equation}

    We may rescale the domain by the dilation $x \mapsto \frac{x}{1 - \epsilon_0}$
    to work on the unit ball $B$, preserving the
    equivalence~\eqref{eqEquivalentMetrics} with a possibly larger constant $C_0$.

    Let $f \in W^{1,p}(M)$ and $\omega \in L^q(M)$ be non-negative with
    $\mathbb{E}[\omega] = 1$. We define $\tilde{f} := f \circ \Psi$, and wish to
    construct a weight $\tilde{\omega}$ on $B$ such that
    $\mathbb{E}[\tilde{\omega}] = 1$. For this, we use the coarea formula (see
    e.g.~\cite[Chapter 5.2]{KrantzParks}):
    \begin{equation}\label{eqCoarea}
        \int_B h(x) \, dx = \int_M \left(\sum_{x \in \Psi^{-1}(y)} \frac{h(x)}{|\det d\Psi_x|} \right) d\mu^g(y),
    \end{equation}
    which holds for any integrable function $h : B \to \bR$.

    Define
    \[
        \tilde{\omega}(x) := \frac{1}{\# \Psi^{-1}(\Psi(x))} |\det d\Psi_x| \cdot \omega(\Psi(x)).
    \]
    Then $\tilde{\omega}$ is measurable, non-negative, and satisfies:
    \[
        \int_B \tilde{\omega}(x) \, dx = \int_M \omega(y) \, d\mu^g(y) = 1.
    \]

    We now estimate the $L^q$-norm of $\tilde{\omega}$. Using
    again~\eqref{eqCoarea}, we compute:
    \begin{align*}
        \int_B \tilde{\omega}(x)^q \, dx
         & = \int_M \left( \sum_{x \in \Psi^{-1}(y)} \frac{\tilde{\omega}(x)^q}{|\det d\Psi_x|} \right) d\mu^g(y)                           \\
         & = \int_M \left( \sum_{x \in \Psi^{-1}(y)} \frac{|\det d\Psi_x|^{q - 1}}{(\# \Psi^{-1}(y))^q} \cdot \omega(y)^q \right) d\mu^g(y) \\
         & \leq \sup_{x \in B} |\det d\Psi_x|^{q - 1} \cdot \int_M \omega(y)^q \, d\mu^g(y).
    \end{align*}
    Therefore,
    \begin{equation}\label{eqEquivalenceLq}
        \|\tilde{\omega}\|_{L^q(B)} \leq C_2 \|\omega\|_{L^q(M)},
    \end{equation}
    where $C_1 = \left( \sup_B |\det d\Psi_x|^{q - 1} \right)^{1/q}$.

    We now compute the $L^p$-norm of the differential $d\tilde{f}$. Since
    $\tilde{f} = f \circ \Psi$, and using the equivalence of
    metrics~\eqref{eqEquivalentMetrics}, we obtain:
    \[
        \| d\tilde{f} \|_{L^p(B)}^p \leq C_3 \| df \|_{L^p(M)}^p,
    \]
    for some constant $C_3 = C_3(\Psi, g, p)$.

    We now apply Lemma~\ref{lmPoincare2} to $\tilde{f}$ and $\tilde{\omega}$, which
    yields:
    \[
        \left\| \tilde{f} - \mathbb{E}_{\tilde{\omega}}[\tilde{f}] \right\|_{L^t(B)} \leq C_4 \|\tilde{\omega}\|_{L^q(B)}^{\frac{n}{(n-1)p}} \| d\tilde{f} \|_{L^p(B)},
    \]
    for some constant $C_4$ depending only on $n, p, q, t$.

    From the construction of $\tilde{\omega}$, a direct application
    of~\eqref{eqCoarea} shows that
    \[
        \mathbb{E}_{\tilde{\omega}}[\tilde{f}] = \mathbb{E}_{\omega}[f].
    \]

    Combining the above inequalities, we obtain:
    \[
        \left\| f - \mathbb{E}_{\omega}[f] \right\|_{L^t(M)} \leq C_5 \left\| \tilde{f} - \mathbb{E}_{\tilde{\omega}}[\tilde{f}] \right\|_{L^t(B)} \leq C \|\omega\|_{L^q(M)}^{\frac{n}{(n-1)p}} \| df \|_{L^p(M)},
    \]
    where $C = C(n, p, q, t, g)$, which completes the proof.
\end{proof}

\providecommand{\bysame}{\leavevmode\hbox to3em{\hrulefill}\thinspace}
\providecommand{\MR}{\relax\ifhmode\unskip\space\fi MR }
\providecommand{\MRhref}[2]{%
    \href{http://www.ams.org/mathscinet-getitem?mr=#1}{#2}
}
\providecommand{\href}[2]{#2}


\end{document}